\documentclass{amsart}
\usepackage{amsmath}
\usepackage{amssymb}
\usepackage{amsfonts}

\newtheorem{theorem}{Theorem}
\theoremstyle{plain}

\newtheorem{corollary}{Corollary}

\newtheorem{definition}{Definition}

\numberwithin{equation}{section}
\input{tcilatex}

\begin{document}
\title[vanishing generalized weighted Morrey spaces]{On the behaviors of
rough fractional type sublinear operators on vanishing generalized weighted
Morrey spaces}
\author{FER\.{I}T G\"{U}RB\"{U}Z}
\address{HAKKARI UNIVERSITY, FACULTY OF EDUCATION, DEPARTMENT OF MATHEMATICS
EDUCATION, HAKKARI, TURKEY }
\email{feritgurbuz@hakkari.edu.tr}
\urladdr{}
\thanks{}
\curraddr{ }
\urladdr{}
\thanks{}
\date{}
\dedicatory{}
\thanks{}

\begin{abstract}
The aim of this paper is to get the boundedness of rough sublinear operators
generated by fractional integral operators on vanishing generalized weighted
Morrey spaces under generic size conditions which are satisfied by most of
the operators in harmonic analysis. Also, rough fractional integral operator
and a related rough fractional maximal operator which satisfy the conditions
of our main result can be considered as some examples.
\end{abstract}

\maketitle

\section{Introduction and useful informations}

\subsection{Background}

The classical fractional integral (The classical fractional integral
operator is also known as Riesz potential.) was introduced by Riesz in 1949 
\cite{Riesz}, defined by 
\begin{eqnarray*}
I_{\alpha }f(x) &=&\left( -\Delta \right) ^{-\frac{\alpha }{2}}f(x)\qquad
0<\alpha <n, \\
&=&\frac{1}{\gamma \left( \alpha \right) }\int \limits_{{\mathbb{R}^{n}}}%
\frac{f(y)}{|x-y|^{n-\alpha }}dy
\end{eqnarray*}%
with 
\begin{equation*}
\gamma \left( \alpha \right) =\frac{\pi ^{\frac{n}{2}}2^{\alpha }\Gamma
\left( \frac{\alpha }{2}\right) }{\Gamma \left( \frac{n}{2}-\frac{\alpha }{2}%
\right) },
\end{equation*}%
where $\Gamma \left( \cdot \right) $ is the standard gamma function and $%
I_{\alpha }$ plays an important role in partial diferential equation as the
inverse of power of Laplace operator. Especially, Its most significant
feature is that $I_{\alpha }$ maps $L_{p}({\mathbb{R}^{n}})$ continuously
into $L_{q}({\mathbb{R}^{n}})$, with $\frac{1}{q}=\frac{1}{p}-\frac{\alpha }{%
n}$ and $1<p<\frac{n}{\alpha }$, through the well known
Hardy-Littlewood-Sobolev imbedding theorem (see pp. 119-121,Theorem 1 and
its proof in \cite{Stein}) for $I_{\alpha }$.

Let $\Omega \in L_{s}(S^{n-1})$, $1<s\leq \infty $, $\Omega (\mu x)=\Omega
(x)~$for any$~\mu >0$, $x\in {\mathbb{R}^{n}}\setminus \{0\}$ and satisfy
the cancellation condition 
\begin{equation*}
\int \limits_{S^{n-1}}\Omega (x^{\prime })d\sigma (x^{\prime })=0,
\end{equation*}%
where $x^{\prime }=\frac{x}{|x|}$ for any $x\neq 0$.

We first recall the definitions of rough fractional integral operator $%
\overline{T}_{\Omega ,\alpha }$ and a related rough fractional maximal
operator $M_{\Omega ,\alpha }$.

\begin{definition}
Define%
\begin{equation*}
I_{\Omega ,\alpha }f(x)=\int \limits_{{\mathbb{R}^{n}}}\frac{\Omega (x-y)}{%
|x-y|^{n-\alpha }}f(y)dy\qquad 0<\alpha <n,
\end{equation*}%
\begin{equation*}
M_{\Omega ,\alpha }f\left( x\right) =\sup_{r>0}\frac{1}{r^{n-\alpha }}\dint
\limits_{|x-y|<r}\left \vert \Omega (x-y)\right \vert \left \vert f(y)\right
\vert dy\qquad 0<\alpha <n.
\end{equation*}
\end{definition}

Next, we give the definition of weighted Lebesgue spaces as follows:

\begin{definition}
$\left( \text{\textbf{Weighted Lebesgue space}}\right) $ Let $1\leq p\leq
\infty $ and given a weight $w\left( x\right) \in A_{p}\left( {{\mathbb{R}%
^{n}}}\right) $, we shall define weighted Lebesgue spaces as 
\begin{eqnarray*}
L_{p}(w) &\equiv &L_{p}({{\mathbb{R}^{n}}},w)=\left \{ f:\Vert f\Vert
_{L_{p,w}}=\left( \dint \limits_{{{\mathbb{R}^{n}}}}|f(x)|^{p}w(x)dx\right)
^{\frac{1}{p}}<\infty \right \} ,\qquad 1\leq p<\infty . \\
L_{\infty ,w} &\equiv &L_{\infty }({{\mathbb{R}^{n}}},w)=\left \{ f:\Vert
f\Vert _{L_{\infty ,w}}=\limfunc{esssup}\limits_{x\in {\mathbb{R}^{n}}%
}|f(x)|w(x)<\infty \right \} .
\end{eqnarray*}
\end{definition}

Here and later, $A_{p}$ denotes the Muckenhoupt classes (see \cite{Gurbuz}).

Now, let us consider the Muckenhoupt-Wheeden class $A\left( p,q\right) $ in 
\cite{Muckenhoupt}. One says that $w\left( x\right) \in A\left( p,q\right) $
for $1<p<q<\infty $ if and only if 
\begin{equation}
\lbrack w]_{A\left( p,q\right) }:=\sup \limits_{B}\left( |B|^{-1}\dint
\limits_{B}w(x)^{q}dx\right) ^{\frac{1}{q}}\left( |B|^{-1}\dint
\limits_{B}w(x)^{-p^{\prime }}dx\right) ^{\frac{1}{p^{\prime }}}<\infty ,
\label{13}
\end{equation}%
where the supremum is taken over all the balls $B$. Note that, by H\"{o}%
lder's inequality, for all balls $B$ we have 
\begin{equation}
\lbrack w]_{A\left( p,q\right) }\geq \lbrack w]_{A\left( p,q\right)
(B)}=|B|^{\frac{1}{p}-\frac{1}{q}-1}\Vert w\Vert _{L_{q}(B)}\Vert
w^{-1}\Vert _{L_{p^{\prime }}(B)}\geq 1.  \label{14}
\end{equation}

By (\ref{13}), we have%
\begin{equation}
\left( \dint \limits_{B}w(x)^{q}dx\right) ^{\frac{1}{q}}\left( \dint
\limits_{B}w(x)^{-p^{\prime }}dx\right) ^{\frac{1}{p^{\prime }}}\lesssim
\left \vert B\right \vert ^{\frac{1}{q}+\frac{1}{p^{\prime }}}.  \label{16}
\end{equation}%
On the other hand, let $\mu \left( x\right) =w\left( x\right) ^{s^{\prime }}$%
, $\tilde{p}=\frac{p}{s^{\prime }}$ and $\tilde{q}=\frac{q}{s^{\prime }}$.
If $w\left( x\right) ^{s^{\prime }}\in A\left( \frac{p}{s^{\prime }},\frac{q%
}{s^{\prime }}\right) $, then we get $\mu \left( x\right) \in A\left( \tilde{%
p},\tilde{q}\right) $. By (\ref{14}) and (\ref{16}), 
\begin{equation}
\Vert \mu \Vert _{L_{\tilde{q}}(B)}\Vert \mu ^{-1}\Vert _{L_{\tilde{p}%
^{\prime }}(B)}\approx |B|^{1+\frac{1}{\tilde{q}}-\frac{1}{\tilde{p}}}
\label{fg}
\end{equation}%
is valid.

Now, we introduce some spaces which play important roles in PDE. Except the
weighted Lebesgue space $L_{p}(w)$, the weighted Morrey space $L_{p,\kappa
}(w)$, which is a natural generalization of $L_{p}(w)$ is another important
function space. Then, the definition of generalized weighted Morrey spaces $%
M_{p,\varphi }\left( w\right) $ which could be viewed as extension of $%
L_{p,\kappa }(w)$ has been given as follows:

For $1\leq p<\infty $, positive measurable function $\varphi (x,r)$ on ${%
\mathbb{R}^{n}}\times (0,\infty )$ and nonnegative measurable function $w$
on ${\mathbb{R}^{n}}$, $f\in M_{p,\varphi }(w)\equiv M_{p,\varphi }({\mathbb{%
R}^{n}},w)$ if $f\in L_{p,w}^{loc}({\mathbb{R}^{n}})$ and 
\begin{equation*}
\Vert f\Vert _{M_{p,\varphi }(w)}=\sup \limits_{x\in {\mathbb{R}^{n}},r>0}%
\frac{1}{\varphi (x,r)}\Vert f\Vert _{L_{p}(B(x,r),w)}<\infty .
\end{equation*}%
is finite. Note that for $\varphi (x,r)\equiv w(B(x,r))^{\frac{\kappa }{p}}$%
, $0<\kappa <1$ and $\varphi (x,r)\equiv 1$, we have $M_{p,\varphi
}(w)=L_{p,\kappa }(w)$ and $M_{p,\varphi }(w)=L_{p}(w)$, respectively.

Extending the definition of vanishing generalized Morrey spaces in \cite%
{Gurbuz1} to the case of generalized weighted Morrey spaces defined above,
we introduce the following definition.

\begin{definition}
\textbf{(Vanishing generalized weighted Morrey spaces) }For\textbf{\ }$1\leq
p<\infty $, $\varphi (x,r)$ is a positive measurable function on ${\mathbb{R}%
^{n}}\times (0,\infty )$ and nonnegative measurable function $w$ on ${%
\mathbb{R}^{n}}$, $f\in VM_{p,\varphi }\left( w\right) \equiv VM_{p,\varphi
}({\mathbb{R}^{n},w})$ if $f\in L_{p,w}^{loc}({\mathbb{R}^{n}})$ and%
\begin{equation}
\lim \limits_{r\rightarrow 0}\sup \limits_{x\in {\mathbb{R}^{n}}}\frac{1}{%
\varphi (x,r)}\Vert f\Vert _{L_{p}(B(x,r),w)}=0.  \label{1*}
\end{equation}
\end{definition}

Inherently, it is appropriate to impose on $\varphi (x,t)$ with the
following circumstances:

\begin{equation}
\lim_{t\rightarrow 0}\sup \limits_{x\in {\mathbb{R}^{n}}}\frac{\left(
w(B(x,t))\right) ^{^{\frac{1}{p}}}}{\varphi (x,t)}=0,  \label{2}
\end{equation}%
and%
\begin{equation}
\inf_{t>1}\sup \limits_{x\in {\mathbb{R}^{n}}}\frac{\left( w(B(x,t))\right)
^{^{\frac{1}{p}}}}{\varphi (x,t)}>0.  \label{3}
\end{equation}%
From (\ref{2}) and (\ref{3}), we easily know that the bounded functions with
compact support belong to $VM_{p,\varphi }\left( w\right) $. On the other
hand, the space $VM_{p,\varphi }(w)$ is Banach space with respect to the
following finite quasi-norm%
\begin{equation*}
\Vert f\Vert _{VM_{p,\varphi }(w)}=\sup \limits_{x\in {\mathbb{R}^{n}},r>0}%
\frac{1}{\varphi (x,r)}\Vert f\Vert _{L_{p}(B(x,r),w)},
\end{equation*}%
such that%
\begin{equation*}
\lim \limits_{r\rightarrow 0}\sup \limits_{x\in {\mathbb{R}^{n}}}\frac{1}{%
\varphi (x,r)}\Vert f\Vert _{L_{p}(B(x,r),w)}=0,
\end{equation*}%
we omit the details. Moreover, we have the following embeddings:%
\begin{equation*}
VM_{p,\varphi }\left( w\right) \subset M_{p,\varphi }\left( w\right) ,\qquad
\Vert f\Vert _{M_{p,\varphi }\left( w\right) }\leq \Vert f\Vert
_{VM_{p,\varphi }\left( w\right) }.
\end{equation*}%
Henceforth, we denote by $\varphi \in \mathcal{B}\left( w\right) $ if $%
\varphi (x,r)$ is a positive measurable function on ${\mathbb{R}^{n}}\times
(0,\infty )$ and positive for all $(x,r)\in {\mathbb{R}^{n}}\times (0,\infty
)$ and satisfies (\ref{2}) and (\ref{3}).

The purpose of this paper is to consider the mapping properties for the
rough fractional type sublinear operators $T_{\Omega ,\alpha }$ satisfying
the following condition%
\begin{equation}
|T_{\Omega ,\alpha }f(x)|\lesssim \int \limits_{{\mathbb{R}^{n}}}\frac{%
|\Omega (x-y)|}{|x-y|^{n-\alpha }}\,|f(y)|\,dy,\qquad x\notin \text{supp }f
\label{e1}
\end{equation}%
on vanishing generalized weighted Morrey spaces. Similar results still hold
for the operators $I_{\Omega ,\alpha }$ and $M_{\Omega ,\alpha }$,
respectively. On the other hand, these operators have not also been studied
so far on vanishing generalized weighted Morrey spaces and this paper seems
to be the first in this direction.

At last, here and henceforth, $F\approx G$ means $F\gtrsim G\gtrsim F$;
while $F\gtrsim G$ means $F\geq CG$ for a constant $C>0$; and $p^{\prime }$
and $s^{\prime }$ always denote the conjugate index of any $p>1$ and $s>1$,
that is, $\frac{1}{p^{\prime }}:=1-\frac{1}{p}$ and $\frac{1}{s^{\prime }}%
:=1-\frac{1}{s}$ and also $C$ stands for a positive constant that can change
its value in each statement without explicit mention. Throughout the paper
we assume that $x\in {\mathbb{R}^{n}}$ and $r>0$ and also let $B(x,r)$
denotes $x$-centred Euclidean ball with radius $r$, $B^{C}(x,r)$ denotes its
complement. For any set $E$, $\chi _{_{E}}$ denotes its characteristic
function, if $E$ is also measurable and $w$ is a weight, $w(E):=\dint
\limits_{E}w(x)dx$.

\section{Main Results}

Our result can be stated as follows.

\begin{theorem}
\label{teo1}Suppose that $0<\alpha <n$, $1\leq s^{\prime }<p<\frac{n}{\alpha 
}$, $\frac{1}{q}=\frac{1}{p}-\frac{\alpha }{n}$, $1<q<\infty $, $\Omega \in
L_{s}(S^{n-1})$, $1<s\leq \infty $, $\Omega (\mu x)=\Omega (x)~$for any$~\mu
>0$, $x\in {\mathbb{R}^{n}}\setminus \{0\}$ such that $T_{\Omega ,\alpha }$
is rough fractional type sublinear operator satisfying (\ref{e1}). For $p>1$%
, $w\left( x\right) ^{s^{\prime }}\in A\left( \frac{p}{s^{\prime }},\frac{q}{%
s^{\prime }}\right) $ and $s^{\prime }<p$, the following pointwise estimate%
\begin{equation}
\left \Vert T_{\Omega ,\alpha }f\right \Vert _{L_{q}\left( B\left(
x_{0},r\right) ,w^{q}\right) }\lesssim \left( w^{q}\left( B\left(
x_{0},r\right) \right) \right) ^{\frac{1}{q}}\int \limits_{2r}^{\infty
}\left \Vert f\right \Vert _{L_{p}\left( B\left( x_{0},t\right)
,w^{p}\right) }\left( w^{q}\left( B\left( x_{0},t\right) \right) \right) ^{-%
\frac{1}{q}}\frac{dt}{t}  \label{5}
\end{equation}%
holds for any ball $B\left( x_{0},r\right) $ and for all $f\in
L_{p,w}^{loc}\left( {\mathbb{R}^{n}}\right) $. If $\varphi _{1}\in \mathcal{B%
}\left( w^{p}\right) $, $\varphi _{2}\in \mathcal{B}\left( w^{q}\right) $
and the pair $\left( \varphi _{1},\varphi _{2}\right) $ satisfies the
following conditions 
\begin{equation}
c_{\delta }:=\dint \limits_{\delta }^{\infty }\sup_{x\in {\mathbb{R}^{n}}}%
\frac{\varphi _{1}\left( x,t\right) }{\left( w^{q}\left( B\left( x,t\right)
\right) \right) ^{\frac{1}{q}}}\frac{1}{t}dt<\infty  \label{6}
\end{equation}%
for every $\delta >0$, and 
\begin{equation}
\int \limits_{r}^{\infty }\frac{\varphi _{1}\left( x,t\right) }{\left(
w^{q}\left( B\left( x,t\right) \right) \right) ^{\frac{1}{q}}}\frac{1}{t}%
dt\lesssim \frac{\varphi _{2}(x,r)}{\left( w^{q}\left( B\left( x,t\right)
\right) \right) ^{\frac{1}{q}}},  \label{7}
\end{equation}

then for $p>1$, $w\left( x\right) ^{s^{\prime }}\in A\left( \frac{p}{%
s^{\prime }},\frac{q}{s^{\prime }}\right) $ and $s^{\prime }<p$, the
operator $T_{\Omega ,\alpha }$ is bounded from $VM_{p,\varphi _{1}}\left(
w^{p}\right) $ to $VM_{q,\varphi _{2}}\left( w^{q}\right) $. Moreover,%
\begin{equation}
\left \Vert T_{\Omega ,\alpha }f\right \Vert _{VM_{q,\varphi _{2}}\left(
w^{q}\right) }\lesssim \left \Vert f\right \Vert _{VM_{p,\varphi _{1}}\left(
w^{p}\right) }\cdot  \label{8}
\end{equation}
\end{theorem}

\begin{proof}
Since inequality (\ref{5}) is the heart of the proof of (\ref{8}), we first
prove (\ref{5}).

For any $x_{0}\in {\mathbb{R}^{n}}$, we write as $f=f_{1}+f_{2}$, where $%
f_{1}\left( y\right) =f\left( y\right) \chi _{B\left( x_{0},2r\right)
}\left( y\right) $, $f_{2}\left( y\right) =f\left( y\right) \chi _{\left(
B\left( x_{0},2r\right) \right) ^{C}}\left( y\right) $, $r>0$ and $\chi
_{B\left( x_{0},2r\right) }$ denotes the characteristic function of $B\left(
x_{0},2r\right) $. Then%
\begin{equation*}
\left \Vert T_{\Omega ,\alpha }f\right \Vert _{L_{q}\left( w^{q},B\left(
x_{0},r\right) \right) }\leq \left \Vert T_{\Omega ,\alpha }f_{1}\right
\Vert _{L_{q}\left( w^{q},B\left( x_{0},r\right) \right) }+\left \Vert
T_{\Omega ,\alpha }f_{2}\right \Vert _{L_{q}\left( w^{q},B\left(
x_{0},r\right) \right) }.
\end{equation*}%
Let us estimate $\left \Vert T_{\Omega ,\alpha }f_{1}\right \Vert
_{L_{q}\left( w^{q},B\left( x_{0},r\right) \right) }$ and $\left \Vert
T_{\Omega ,\alpha }f_{2}\right \Vert _{L_{q}\left( w^{q},B\left(
x_{0},r\right) \right) }$, respectively.

Since $f_{1}\in L_{p}\left( w^{p},{\mathbb{R}^{n}}\right) $, by the
boundedness of $T_{\Omega ,\alpha }$ from $L_{p}\left( w^{p},{\mathbb{R}^{n}}%
\right) $ to $L_{q}\left( w^{q},{\mathbb{R}^{n}}\right) $ (see Theorem 3.4.2
in \cite{Lu}), (\ref{fg}) and since $1\leq s^{\prime }<p<q$ we get 
\begin{eqnarray*}
\left \Vert T_{\Omega ,\alpha }f_{1}\right \Vert _{L_{q}\left( w^{q},B\left(
x_{0},r\right) \right) } &\leq &\left \Vert T_{\Omega ,\alpha }f_{1}\right
\Vert _{L_{q}\left( w^{q},{\mathbb{R}^{n}}\right) } \\
&\lesssim &\left \Vert f_{1}\right \Vert _{L_{p}\left( w^{p},{\mathbb{R}^{n}}%
\right) } \\
&=&\left \Vert f\right \Vert _{L_{p}\left( w^{p},B\left( x_{0},2r\right)
\right) } \\
&\lesssim &r^{n-\alpha s^{\prime }}\left \Vert f\right \Vert _{L_{p}\left(
w^{p},B\left( x_{0},2r\right) \right) }\dint \limits_{2r}^{\infty }\frac{dt}{%
t^{n-\alpha s^{\prime }+1}} \\
&\approx &\Vert w^{s^{\prime }}\Vert _{L_{\frac{q}{s^{\prime }}}(B\left(
x_{0},r\right) )}\Vert w^{-s^{\prime }}\Vert _{L_{\left( \frac{p}{s^{\prime }%
}\right) ^{\prime }}(B\left( x_{0},r\right) )}\dint \limits_{2r}^{\infty
}\left \Vert f\right \Vert _{L_{p}\left( w^{p},B\left( x_{0},t\right)
\right) }\frac{dt}{t^{n-\alpha s^{\prime }+1}} \\
&\lesssim &\left( w^{q}\left( B(x_{0},r)\right) \right) ^{\frac{1}{q}}\dint
\limits_{2r}^{\infty }\left \Vert f\right \Vert _{L_{p}\left( w^{p},B\left(
x_{0},t\right) \right) }\Vert w^{-s^{\prime }}\Vert _{L_{\left( \frac{p}{%
s^{\prime }}\right) ^{\prime }}(B\left( x_{0},t\right) )}\frac{dt}{%
t^{n-\alpha s^{\prime }+1}} \\
&\lesssim &\left( w^{q}\left( B(x_{0},r)\right) \right) ^{\frac{1}{q}}\dint
\limits_{2r}^{\infty }\left \Vert f\right \Vert _{L_{p}\left( w^{p},B\left(
x_{0},t\right) \right) }\left[ \Vert w^{s^{\prime }}\Vert _{L_{\left( \frac{q%
}{s^{\prime }}\right) }(B\left( x_{0},t\right) )}\right] ^{-1}\frac{1}{t}dt
\\
&\lesssim &\left( w^{q}\left( B(x_{0},r)\right) \right) ^{\frac{1}{q}} \\
&&\times \dint \limits_{2r}^{\infty }\left \Vert f\right \Vert _{L_{p}\left(
w^{p},B(x_{0},t)\right) }\left( w^{q}\left( B(x_{0},t)\right) \right) ^{-%
\frac{1}{q}}\frac{1}{t}dt.
\end{eqnarray*}%
Now, let's estimate the second part ($=\left \Vert T_{\Omega ,\alpha
}f_{2}\right \Vert _{L_{q}\left( w^{q},B\left( x_{0},r\right) \right) }$).
For the estimate used in $\left \Vert T_{\Omega ,\alpha }f_{2}\right \Vert
_{L_{q}\left( w^{q},B\left( x_{0},r\right) \right) }$, we first have to
prove the below inequality:%
\begin{equation}
\left \vert T_{\Omega ,\alpha }f_{2}\left( x\right) \right \vert \lesssim
\dint \limits_{2r}^{\infty }\left \Vert f\right \Vert _{L_{p}\left(
w^{p},B(x_{0},t)\right) }\left( w^{q}\left( B(x_{0},t)\right) \right) ^{-%
\frac{1}{q}}\frac{1}{t}dt.  \label{11}
\end{equation}

By \cite{Balak} (see pp. 7 in the proof of Lemma2:), we get%
\begin{equation}
\left \vert T_{\Omega ,\alpha }f_{2}\left( x\right) \right \vert \lesssim
\int \limits_{2r}^{\infty }\left \Vert \Omega \left( x-\cdot \right) \right
\Vert _{L_{s}\left( B\left( x_{0},t\right) \right) }\left \Vert f\right
\Vert _{L_{s^{\prime }}\left( B\left( x_{0},t\right) \right) }\frac{dt}{%
t^{n+1-\alpha }}.  \label{310}
\end{equation}%
On the other hand, by H\"{o}lder's inequality we have%
\begin{eqnarray}
\left \Vert f\right \Vert _{L_{s^{\prime }}\left( B\left( x_{0},t\right)
\right) } &=&\left( \dint \limits_{B\left( x_{0},t\right) }\left \vert
f\left( y\right) \right \vert ^{s^{\prime }}dy\right) ^{\frac{1}{s^{\prime }}%
}  \notag \\
&\leq &\left( \dint \limits_{B\left( x_{0},t\right) }\left \vert f\left(
y\right) \right \vert ^{p}\left \vert \mu \left( y\right) \right \vert ^{%
\widetilde{p}}dy\right) ^{\frac{1}{p}}\left( \dint \limits_{B\left(
x_{0},t\right) }\left \vert \mu \left( y\right) \right \vert ^{-\widetilde{p}%
^{\prime }}dy\right) ^{\frac{1}{\widetilde{p}^{\prime }s^{\prime }}}  \notag
\\
&\leq &\left( \dint \limits_{B\left( x_{0},t\right) }\left \vert f\left(
y\right) \right \vert ^{p}\left \vert \mu \left( y\right) \right \vert ^{%
\widetilde{p}}dy\right) ^{\frac{1}{p}}\left( w^{q}\left( B(x_{0},t)\right)
\right) ^{-\frac{1}{q}}|B(x_{0},t)|^{\frac{1}{s^{\prime }}+\frac{1}{q}-\frac{%
1}{p}}  \notag \\
&=&\left \Vert f\right \Vert _{L_{p}\left( w^{p},B(x_{0},t)\right) }\left(
w^{q}\left( B(x_{0},t)\right) \right) ^{-\frac{1}{q}}|B(x_{0},t)|^{\frac{1}{%
s^{\prime }}+\frac{1}{q}-\frac{1}{p}},  \label{36}
\end{eqnarray}%
where in the second inequality we have used the following fact:

By (\ref{fg}), we get the following:%
\begin{eqnarray}
\left( \dint\limits_{B\left( x_{0},t\right) }\left\vert \mu \left( y\right)
\right\vert ^{-\widetilde{p}^{\prime }}dy\right) ^{\frac{1}{\widetilde{p}%
^{\prime }s^{\prime }}} &\approx &\left[ \Vert \mu \Vert _{L_{\tilde{q}%
}(B\left( x_{0},t\right) )}\right] ^{-\frac{1}{s^{\prime }}}\left[ |B\left(
x_{0},t\right) |^{1+\frac{1}{\tilde{q}}-\frac{1}{\tilde{p}}}\right] ^{\frac{1%
}{s^{\prime }}}  \notag \\
&=&\left[ \left( \Vert w^{s^{\prime }}\Vert _{L_{\tilde{q}}(B\left(
x_{0},t\right) )}\right) ^{-1}|B\left( x_{0},t\right) |^{1+\frac{1}{\tilde{q}%
}-\frac{1}{\tilde{p}}}\right] ^{\frac{1}{s^{\prime }}}  \notag \\
&=&\left[ \left( \dint\limits_{B\left( x_{0},t\right) }\left\vert w\left(
y\right) \right\vert ^{q}dy\right) ^{-\frac{s^{\prime }}{q}}|B\left(
x_{0},t\right) |^{1+\frac{s^{\prime }}{q}-\frac{s^{\prime }}{p}}\right] ^{%
\frac{1}{s^{\prime }}}  \notag \\
&=&\left( w^{q}\left( B(x_{0},t)\right) \right) ^{-\frac{1}{q}}|B(x_{0},t)|^{%
\frac{1}{s^{\prime }}+\frac{1}{q}-\frac{1}{p}}.  \label{100}
\end{eqnarray}%
At last, substituting (3.10) in \cite{Balak} and (\ref{36}) into (\ref{310}%
), the proof of (\ref{11}) is completed. Thus, by (\ref{11}) we get 
\begin{eqnarray*}
\left\Vert T_{\Omega ,\alpha }f_{2}\right\Vert _{L_{q}\left( w^{q},B\left(
x_{0},r\right) \right) } &\lesssim &\left( w^{q}\left( B(x_{0},r)\right)
\right) ^{\frac{1}{q}} \\
&&\times \dint\limits_{2r}^{\infty }\left\Vert f\right\Vert _{L_{p}\left(
w^{p},B(x_{0},t)\right) }\left( w^{q}\left( B(x_{0},t)\right) \right) ^{-%
\frac{1}{q}}\frac{1}{t}dt.
\end{eqnarray*}%
Combining all the estimates for $\left\Vert T_{\Omega ,\alpha
}f_{1}\right\Vert _{L_{q}\left( w^{q},B\left( x_{0},r\right) \right) }$ and $%
\left\Vert T_{\Omega ,\alpha }f_{2}\right\Vert _{L_{q}\left( w^{q},B\left(
x_{0},r\right) \right) }$, we get (\ref{5}).

Now, let's estimate the second part (\ref{8}) of Theorem \ref{teo1}. Indeed,
by the definition of vanishing generalized weighted Morrey spaces, (\ref{5})
and (\ref{7}), we have%
\begin{eqnarray*}
\Vert T_{\Omega ,\alpha }f\Vert _{VM_{q,\varphi _{2}}\left( w^{q}\right) }
&=&\sup \limits_{x\in {\mathbb{R}^{n}},r>0}\frac{1}{\varphi _{2}(x,r)}\Vert
T_{\Omega ,\alpha }f\Vert _{L_{q}\left( w^{q},B\left( x_{0},r\right) \right)
} \\
&\lesssim &\sup \limits_{x\in {\mathbb{R}^{n}},r>0}\frac{1}{\varphi _{2}(x,r)%
}\left( w^{q}\left( B\left( x_{0},r\right) \right) \right) ^{\frac{1}{q}} \\
&&\times \int \limits_{r}^{\infty }\left \Vert f\right \Vert _{L_{p}\left(
B\left( x_{0},t\right) ,w^{p}\right) }\left( w^{q}\left( B\left(
x_{0},t\right) \right) \right) ^{-\frac{1}{q}}\frac{dt}{t} \\
&\lesssim &\sup \limits_{x\in {\mathbb{R}^{n}},r>0}\frac{1}{\varphi _{2}(x,r)%
}\left( w^{q}\left( B\left( x_{0},r\right) \right) \right) ^{\frac{1}{q}} \\
&&\times \int \limits_{r}^{\infty }\left( w^{q}\left( B\left( x_{0},t\right)
\right) \right) ^{-\frac{1}{q}}\varphi _{1}\left( x,t\right) \left[ \varphi
_{1}\left( x,t\right) ^{-1}\left \Vert f\right \Vert _{L_{p}\left( B\left(
x_{0},t\right) ,w^{p}\right) }\right] \frac{dt}{t} \\
&\lesssim &\left \Vert f\right \Vert _{VM_{p,\varphi _{1}}\left(
w^{p}\right) }\sup \limits_{x\in {\mathbb{R}^{n}},r>0}\frac{1}{\varphi
_{2}(x,r)}\left( w^{q}\left( B\left( x_{0},r\right) \right) \right) ^{\frac{1%
}{q}} \\
&&\times \int \limits_{r}^{\infty }\left( w^{q}\left( B\left( x_{0},t\right)
\right) \right) ^{-\frac{1}{q}}\varphi _{1}\left( x,t\right) \frac{dt}{t} \\
&\lesssim &\left \Vert f\right \Vert _{VM_{p,\varphi _{1}}\left(
w^{p}\right) }.
\end{eqnarray*}%
At last, we need to prove that%
\begin{equation*}
\lim \limits_{r\rightarrow 0}\sup \limits_{x\in {\mathbb{R}^{n}}}\frac{1}{%
\varphi _{2}(x,r)}\Vert T_{\Omega ,\alpha }f\Vert _{L_{q}\left(
w^{q},B\left( x_{0},r\right) \right) }\lesssim \lim \limits_{r\rightarrow
0}\sup \limits_{x\in {\mathbb{R}^{n}}}\frac{1}{\varphi _{1}(x,r)}\Vert
f\Vert _{L_{p}\left( w^{p},B\left( x_{0},r\right) \right) }=0.
\end{equation*}%
But, because the proof of above inequality is similar to Theorem 2 in \cite%
{Gurbuz1}, we omit the details, which completes the proof.
\end{proof}

\begin{corollary}
Under the conditions of Theorem \ref{teo1}, the operators $M_{\Omega ,\alpha
}$ and $I_{\Omega ,\alpha }$ are bounded from $VM_{p,\varphi _{1}}\left(
w^{p}\right) $ to $VM_{q,\varphi _{2}}\left( w^{q}\right) $.
\end{corollary}

\begin{corollary}
For $w\equiv 1$, under the conditions of Theorem \ref{teo1}, we get the
Theorem 2 in \cite{Gurbuz1}.
\end{corollary}

\end{document}